\documentclass[12pt,leqno,twoside]{amsart}
\usepackage{amssymb, tikz-cd, enumitem} 
\usepackage{latexsym}
\usepackage{graphicx}
\usepackage[colorlinks=true, pdfstartview=FitV, linkcolor=blue, citecolor=blue, urlcolor=blue]{hyperref}
\usepackage{subcaption} 
\usepackage{mathtools, dsfont}
\usepackage{helvet}
\usepackage[margin=1in]{geometry}

\newtheorem{thm}{Theorem}[section]
\newtheorem{cor}[thm]{Corollary}
\newtheorem{defn}[thm]{Definition}
\newtheorem{example}[thm]{Example}

\newtheorem{lemma}[thm]{Lemma}
\newtheorem{prop}[thm]{Proposition}
\newtheorem{remark}[thm]{Remark}
\newtheorem{conj}[thm]{Conjecture}

\newtheorem{quest}[thm]{Question}

\newtheorem{question}[thm]{Question}

\numberwithin{equation}{section}

\newcommand{\Z}{\mathbb{Z}}

\newcommand{\Q}{\mathbb{Q}}
\newcommand{\R}{\mathbb{R}}
\newcommand{\C}{\mathbb{C}}
\newcommand{\K}{\mathbb{K}}

\renewcommand{\H}{{H}}

\def\part{\partial}

\def\bd{\begin{defn}}
\def\ed{\end{defn}}
\def\bt{\begin{thm}}
\def\et{\end{thm}}
\def\br{\begin{remark}}
\def\er{\end{remark}}
\def\bc{\begin{cor}}
\def\ec{\end{cor}}
\def\bp{\begin{prop}}
\def\ep{\end{prop}}
\def\be{\begin{equation}}
\def\ee{\end{equation}}
\def\bn{\begin{enumerate}}
\def\en{\end{enumerate}}
\def\ba{\begin{array}}
\def\ea{\end{array}}
\def\bex{\begin{example}}
\def\eex{\end{example}}

\newcommand\sO{{\mathcal O}}

\newcommand\sH{{\mathcal H}}

\newcommand\sP{{\mathcal P}}
\newcommand\sI{{\mathcal I}}

\newcommand\sF{{\mathcal F}}

\newcommand\sM{{\mathcal M}}

\newcommand\sL{\mathcal{L}}

\def\bZ{\mathbb{Z}}
\def\bZ{\mathbb{Z}}

\def\Lotimes{\otimes^L}
\def\dR{\mathds{A}}

\title[Aspherical manifolds]{Aspherical manifolds, Mellin transformation \\ and a question of  Bobadilla-Koll\'{a}r}

\begin{document}

\author[Y. Liu ]{Yongqiang Liu}
\address{Y. Liu : The Institute of Geometry and Physics, University of Science and Technology of China, 96 Jinzhai Road, Hefei 230026 P.R. China}
\email{liuyq@ustc.edu.cn}

\author[L. Maxim ]{Lauren\c{t}iu Maxim}
\address{L. Maxim : Department of Mathematics, University of Wisconsin-Madison, 480 Lincoln Drive, Madison WI 53706-1388, USA}
\email {maxim@math.wisc.edu}

\author[B. Wang]{Botong Wang}
\address{B. Wang : Department of Mathematics, University of Wisconsin-Madison, 480 Lincoln Drive, Madison WI 53706-1388, USA}
\email {wang@math.wisc.edu}



\date{\today}

\begin{abstract} 
In their 2012 paper, Bobadilla and Koll\'ar studied topological conditions which guarantee that a proper map of complex algebraic varieties is a topological or differentiable fibration. They also asked whether a certain finiteness property on the relative covering space can imply that a proper map is a fibration. In this paper, we answer positively the integral homology version of their question in the case of abelian varieties, and the rational homology version in the case of compact ball quotients. We also propose several conjectures in relation to the Singer-Hopf conjecture in the complex projective setting. 
\end{abstract}

\maketitle

\section{Introduction}
A CW-complex $X$ is called {\it aspherical} if it is connected and all its higher homotopy groups vanish, i.e., $\pi_i(X)$ is trivial for all $i \geq 2$. The vanishing of higher-homotopy groups is equivalent to the fact that the universal covering $\widetilde{X}$ of $X$ is contractible. The homotopy type of an aspherical CW complex depends only on its fundamental group (e.g., see \cite{Lu10} for a survey).

Interesting examples of aspherical spaces are the closed Riemannian manifolds with non-positive sectional curvature, compact ball quotients or abelian varieties.

There are several prominent open conjectures concerning aspherical manifolds. For instance, a conjecture of Borel asserts that two aspherical closed manifolds are homeomorphic if and only if their fundamental groups are isomorphic. The Borel conjecture is proved in many important cases, e.g., it is true in dimensions $\neq 3, 4$ for all non-positively curved closed Riemannian manifolds, see \cite{FJ}.  Another important conjecture  was made by Hopf (and later on strengthened by Singer) on the sign of the topological Euler characteristic of an aspherical closed manifold.
\begin{conj}{\rm (Singer-Hopf)}\label{SHi} If $X^{2n}$ is a closed, aspherical manifold of real dimension $2n$, then $$(-1)^n  \chi(X^{2n}) \geq 0.$$ \end{conj}

More recently, inspired by work of Koll\'ar and Pardon \cite{KP}, Bobadilla and Koll\'ar \cite{BK} used aspherical manifolds in their search for homotopy/homology fiber bundles, which in turn are conjectured to be differentiable fibre bundles.
In order to state their conjecture, we make the following definition.

\bd[{\cite[Definition 1]{BK}}]\rm
Let $X$ and $Y$ be complex manifolds.\footnote{In \cite{BK}, the authors considered more generally complex spaces. In this paper, we will restrict ourselves to smooth manifolds/varieties.} A proper holomorphic map $f : X \to Y$ is said to be a {\it homotopy fiber bundle} if $Y$ has an open cover $Y = \bigcup_j U_j$ such that for every $j$ and for every $y \in U_j$ the inclusion
\begin{center}
$f^{-1}(y) \hookrightarrow f^{-1}(U_j) $ is a homotopy equivalence.
\end{center}
Similarly, given a commutative ring $\dR$, the map $f : X \to Y$ is called an {\it $\dR$-homology fiber bundle} if
\begin{center}
$\H_*(f^{-1}(y),\dR) \to \H_*( f^{-1}(U_j),\dR) $ 
is an isomorphism. 
\end{center}
\ed 

Let $X$ and $Y$ be  smooth projective varieties and $f : X \to Y$ a surjective morphism.
Let $\widetilde{Y} \to Y$ denote the universal cover. By pull-back we obtain a map $\widetilde{f}: \widetilde{X} \to \widetilde{Y}$.
In \cite{BK}, Bobadilla and Koll\'{a}r asked the following question, which also appears as Question 26 in \cite{KP} (see also \cite[Question 4]{KP} for a broader statement).

\begin{quest}\label{qu}
Assume that $\widetilde{Y}$ is contractible and $\widetilde{X}$ is homotopy equivalent to a finite CW
complex. Does this imply that $f$ is a topological or differentiable fiber bundle?
\end{quest}

The above question can be divided into two parts. The first part is more topological:
\begin{question} {\rm (\cite[Question 4.2]{BK})} \label{q4}
Assume that $\widetilde{Y}$ is contractible and $\widetilde{X}$ is homotopy equivalent to a finite CW
complex. Does this imply that $f$ is a homotopy or $\Z$-homology fiber bundle?
\end{question}
We will refer to the homological part of Question \ref{q4} as the \textit{integral Bobadilla-Koll\'ar question}. If we replace ``$\Z$-homology fiber bundle'' by ``$\Q$-homology fiber bundle'' in the above question, we call it the \textit{rational Bobadilla-Koll\'ar question}. 

The second part is more geometric, and is formulated as a conjecture in \cite{BK}. 
\begin{conj}{\rm (\cite[Conjecture 3]{BK})}
Let $f: X\to Y$ be a proper map of smooth complex algebraic varieties. If $f$ is a homotopy or $\Z$-homology fiber bundle, then it is a differentiable fiber bundle. 
\end{conj}

In this paper we answer positively the homological versions of the Bobadilla-Koll\'ar question \ref{q4} in the case of aspherical projective manifolds with ample cotangent bundles (e.g., compact ball quotients) and abelian varieties. More precisely, we show the following 
(see Theorem \ref{th23} and \ref{th38}).
\bt\label{ti}
The rational Bobadilla-Koll\'ar question has a positive answer  if $Y$ is an aspherical projective manifold with ample cotangent bundle. Moreover, the integral Bobadilla-Koll\'ar question has a positive answer  if $Y$ is an  abelian variety. 
\et

As a concrete application, we get the following.
\begin{cor}
Let $X$ be a projective manifold, and denote by $X^{ab}$ the universal free abelian cover of $X$, i.e., the covering associated to the homomorphism $\pi_1(X)\to H_1(X, \Z)/\mathrm{torsion}$. If $X^{ab}$ is homotopy equivalent to a finite CW-complex, then the Albanese map of $X$ is a $\Z$-homology fiber bundle. 
\end{cor}

Our main results should be compared to \cite[Theorem 16,  Theorem 20]{KP} by Koll\'ar and Pardon, where special cases of Question \ref{qu} are addressed. More precisely, they gave positive answers to the question when $f$ is generically finite, or when $\widetilde{X}$ is an open semialgebraic subset of a projective variety.

Our approach to proving Theorem \ref{ti} relies on the theory of perverse sheaves and derived calculus (for an introduction to these techniques, see, e.g., \cite{Di2, Ma}).
Let us mention here the main ideas of the proof.
First note that a proper map $f:X \to Y$ of smooth complex algebraic varieties is an $\dR$-homology fiber bundle if and only if the sheaves $R^if_*\dR_X$ are locally constant on $Y$ for all $i \geq 0$. If this is the case, we say that the constructible complex $Rf_*\dR_X$ is locally constant. Such complexes are characterized in Section \ref{locc}. 

To answer the homological version of Question \ref{q4}, in Section \ref{mel} we first introduce and study the properties of a nonabelian version of the Mellin transformation considered in \cite{LMWc,LMWd}, see Definition~\ref{def6}. Mellin transformations are the topological counterparts of the more classical Fourier-Mukai transforms in algebraic geometry.

A positive answer to the rational Bobadilla--Koll\'ar question can be given for any aspherical projective manifold with an ample cotangent bundle (e.g., a compact ball quotient) by using the decomposition theorem and positivity results for Chern classes of ample vector bundles (cf. Section \ref{nef}).
For the integral version, we have to also work with fields of positive characteristics. A positive answer to the integral Bobadilla--Koll\'ar question for abelian varieties is given in Section \ref{sec5}, and it relies on a key non-vanishing property of the Mellin transformation (see Proposition \ref{prop_nonzero}), which is proved using  characteristic cycles and the geometry of abelian varieties. 

Finally, in Section \ref{conj} we speculate around the Singer-Hopf conjecture \ref{SHi} in the complex algebraic setting. We propose various generalizations, also in relation to the Shafarevich conjecture. 
Recall that a complex space is holomorphically convex if it admits a proper surjective holomorphic mapping onto some Stein space, which induces an isomorphism between the algebras of holomorphic functions on these spaces. The Shafarevich conjecture then predicts that the universal covering space of a complex projective manifold should be holomorphically convex.
In Section \ref{conj} we prove the following result (see Corollary \ref{cor_twoconj}). 
\bt
Let $Y$ be an aspherical projective manifold. Then the Shafarevich conjecture implies that the universal cover of $Y$ is Stein. 
\et
We also conjecture that the cotangent bundle of a projective manifold with a Stein universal cover is nef (see Conjecture \ref{conj1}).  Together with semi-positivity results for nef vector bundles from \cite{DPS}, this would then imply the Singer-Hopf Conjecture \ref{SHi} in the complex projective setting.

\medskip

As a convention, in this paper all varieties and manifolds are assumed to be connected. 

\medskip

\noindent {\bf Acknowledgements.} This project was started while the second and third authors visited the Basque Center for Applied Mathematics in Bilbao, Spain. We thank this institute for hospitality and for providing us with excellent working conditions. We also thank Donu Arapura and Pierre Py for useful discussions. Finally, we thank the anonymous referee for constructive comments which helped us streamline some of the arguments.
Y. Liu is partially supported by National Key Research and Development Project SQ2020YFA070080, the starting grant KY2340000123 from University of Science and Technology of China, National Natural Science Funds of China (Grant No. 12001511), the Project of Stable Support for Youth Team in Basic Research Field, CAS (YSBR-001),  the project ``Analysis and Geometry on Bundles" of Ministry of Science and Technology of the People's Republic of China and  Fundamental Research Funds for the Central Universities.
L.  Maxim is  partially  supported  by the Simons Foundation  (Collaboration Grant  \#567077),  and by the Romanian Ministry of National Education (CNCS-UEFISCDI grant PN-III-P4-ID-PCE-2020-0029). B. Wang is partially supported by the NSF grant DMS-1701305 and by a Sloan Fellowship.


\section{Nonabelian Mellin transformation}\label{mel}

In this section we introduce the Mellin transformation for complex manifolds and study its immediate properties. 
This is a non-abelian version of the Gabber-Loeser Mellin transformation defined in \cite{GL} for complex affine tori, and further studied in \cite{BSS} for abelian varieties. Here we extend these references to spaces with non-abelian fundamental groups. See \cite{LMWe} for more recent developments.

Let $Y$ be a complex manifold of dimension $d$ with a base point $y_0$. Denote the  fundamental group $\pi_1(Y, y_0)$ by $G$. 
We denote the universal cover of $Y$ by $\widetilde{Y}$. Let $\dR$ be a commutative ring, and let $\sF$ be a bounded $\dR$-constructible complex on $Y$. 
\bd\label{def6}\rm
We define the {\it Mellin transformation} of $\sF$ on $Y$ as
\[
\sM_!(Y, \sF)\coloneqq R\widetilde{q}_! \big(p^*\sF\big)\in D^b(\dR[G]),
\]
where $p: \widetilde{Y}\to Y$ is the universal covering map, $\widetilde{q}: \widetilde{Y}\to \textrm{pt}$ is the projection to a point, and $D^b(\dR[G])$ is the derived category of right $\dR[G]$-modules. When there is no risk of confusion, we will simply write $\sM_!(\sF)$ instead of $\sM_!(Y, \sF)$.
\ed
Note that since $p^*\sF$ is a right $G$-equivariant complex and $\widetilde{q}$ is a right $G$-equivariant map, $R\widetilde{q}_! (p^*\sF)$ admits a natural right $G$-action. 
By definition, the Mellin transformation
\[
\sM_!(Y, -): D^b_c(Y, \dR)\to D^b(\dR[G])
\]
is a functor of triangulated categories. 

Let $\sL_G$ be the rank one $\dR[G]$-local system on $Y$ whose monodromy action is given by multiplication on the right. Note that  $\sL_G\cong Rp_!\dR_{\widetilde{Y}}$ as sheaves of right $\dR[G]$-modules. Then an equivalent description of the Mellin transformation is given by the following
\begin{prop}
\begin{equation}\label{eq_def2}
\sM_!(Y, \sF)= R{q}_!(\sF\otimes_\dR \sL_G)\in D^b(\dR[G]),
\end{equation}
where $q: Y\to \textrm{pt}$ is the projection to a point. 
\end{prop}
\begin{proof}
Indeed, we have:
\[
R\widetilde{q}_! \big(p^*\sF\big)
\cong Rq_! Rp_! \big( p^*\sF\otimes_\dR \dR_{\widetilde{Y}} \big) \cong Rq_! \big( \sF \otimes_\dR Rp_!\dR_{\widetilde{Y}}  \big)\cong Rq_! \big( \sF \otimes_\dR \sL_G \big),
\]
where the second isomorphism uses the projection formula.
\end{proof}

\begin{remark}\rm
Here we made a choice that the equivariant $G$-actions are on the right. This essentially depends on how we let $G$ act on $\widetilde{Y}$. We consider points in $\widetilde{Y}$ as homotopy classes of paths from an arbitrary point $y$ on $Y$ to the base point $y_0$. Then the natural action of $G$ on $\widetilde{Y}$ is on the right. 
\end{remark}

\begin{example}\rm\label{ex9}
Suppose that $Y$ is an aspherical complex manifold of complex dimension $d$ and $L$ is an $\dR$-local system on $Y$. Then $\sM_!(L)\cong V[-2d]$, where $V$ is the $\dR[G]$-module associated to the monodromy representation of $L$. 
\end{example}

Let $V$ be a finitely generated free $\dR$-module. Fix a representation $\rho: G\to \mathrm{Aut}_\dR(V)$, which induces a left $\dR[G]$-module structure on $V$. We write $V_\rho$ to emphasize the $\dR[G]$-module structure on $V$. We denote by $L_\rho$ the $\dR$-local system on $Y$ whose stalks are isomorphic to $V$ and whose monodromy action is equal to $\rho$. 
\begin{prop}\label{prop_iso1}
Given any $\rho$ as above and $\sF\in D^b_c(Y, \dR)$, there is a canonical isomorphism
\be\label{re}
\H^i_c(Y,  \sF\otimes_\dR L_\rho) \cong \H^i\big(\sM_!(Y,\sF) \Lotimes_{\dR[G]}V_{\rho}\big)
\ee
where $\Lotimes_{\dR[G]}$ denotes the derived tensor product of right and left $\dR[G]$-modules. 
\end{prop}
\begin{proof}
By projection formula, we have
$$R{q}_!\big(\sF \otimes_\dR \sL_G\big) \Lotimes_{\dR[G]}V_\rho \cong R{q}_!\big( \sF \otimes_\dR \sL_G \Lotimes_{\dR[G]} q^*V_\rho\big) \cong R{q}_!\big( \sF \otimes_\dR L_\rho\big).$$
Formula \eqref{re} follows by taking $i$-th cohomology on both sides. 
\end{proof}

\bc\label{cor_nonzero} Let $Y$ be a smooth complex algebraic variety or a compact complex manifold. 
Given a field $\K$, if $\sF\in D^b_c(Y, \K)$ has the property that its Euler characteristic with compact support $\chi_c(Y,\sF)$ is not zero, then $\sM_!(Y, \sF)\neq 0$.
\ec
\begin{proof}
Suppose that $\sM_!(Y, \sF)=0$. Applying the above proposition to the case when $V_\rho$ is the trivial rank one $G$-representation, we get that $\H_c^k(Y, \sF)=0$ for all $k\in \Z$. Then $\chi_c(Y, \sF)=0$, contradicting our assumptions. 
\end{proof}
Note that in the algebraic case one has the equality 
$\chi(Y, \sF)=\chi_c(Y, \sF)$, e.g., see \cite[Proposition 4.1.23]{Di2}. Hence Corollary \ref{cor_nonzero} could also be formulated in terms of $\chi(Y, \sF)$.

\begin{cor}\label{cor_allzero}
Let $\sF$ be a $\K$-constructible complex on an abelian variety $A$. Then $\sM_!(\sF)=0$ if and only if $\H^i(A, \sF\otimes_\K L)=0$ for any $i$ and any rank one $\overline{\K}$-local system $L$, where $\overline{\K}$ is the algebraic closure of $\K$. 
\end{cor}
\begin{proof}
Notice that the group ring $\K[\pi_1(A)]$ is isomorphic to a Laurent polynomial ring. Over a finitely generated $\K$-algebra, a finitely generated module is zero if and only if its restriction to every $\overline{\K}$-point is zero. So the assertion follows from Proposition \ref{prop_iso1}. \end{proof}


\section{Euler characteristic of perverse sheaves and characteristic cycles}\label{nef}
In this section, we recall several (semi-)positivity results for ample (resp., nef) vector bundles on projective manifolds. We use such results to deduce (semi-)positivity statements for the Euler characteristics of perverse sheaves on complex projective manifolds with ample (resp., nef) cotangent bundles. For perverse sheaves we use field coefficients.

First, let us recall the definition of ample resp. nef vector bundles. 
\bd\rm
If $E$ is a vector bundle on a  projective manifold $X$, denote by $\mathbf{P}(E)$ the projective bundle of hyperplanes in the fibers of $E$. 
A vector bundle $E$ on $X$ is called {\it ample} (resp. {\it nef}) if the line bundle $\mathcal{O}_E(1)$ on $\mathbf{P}(E)$ is ample (resp. nef). 
\ed

In \cite{FuLa}, Fulton and Lazarsfeld studied the positivity of Chern classes of ample vector bundles, and proved the following result. 
\begin{thm}{\rm (\cite[Theorem II]{FuLa})}
\label{thm_positive}
Let $X$ be a  projective manifold and let $E$ be a rank $r$ ample vector bundle on $X$. For any $r$-dimensional conic subvariety $C$ of $E$, that is a subvariety that is invariant under the $\C^*$-action on $E$, the intersection number satisfies
\[
\langle C, Z_E\rangle_E>0
\]
where $Z_E$ is the zero section of $E$. 
\end{thm}

In the above theorem, even though $E$ is not compact, the intersection number $\langle C, Z_E\rangle_E$ is well-defined. In fact, we can take any smooth compactification of $E$, and replace $C$ by its closure. The intersection number will not depend on the choice of the compactification. For more details we refer to \cite[Section 8.1.B]{La}. 

Together with Kashiwara's index theorem, we have the following positivity result on the Euler characteristics of perverse sheaves. 
\bp\label{pr1}
Let $X$ be a  projective manifold with ample cotangent bundle, and let $\sP$ be a nonzero perverse sheaf on $X$. Then $\chi(X, \sP)>0$. 
\ep

\begin{proof}
Kashiwara's global index theorem \cite{Kas} computes the Euler characteristic of any bounded constructible complex $\sP$ on $X$ by the formula:
\be \chi(X,\mathcal{P})=\langle CC(\sP) , [X]\rangle_{T^*X}, \ee
that is, the intersection index  in the cotangent bundle $T^*X$, of the {characteristic cycle} of $\mathcal{P}$ with the zero section of $T^*X$. Recall that the characteristic cycle $CC(\sP)$ is a formal $\bZ$-linear combination of irreducible conic Lagrangian cycles $T^*_ZX:=\overline{T^*_{Z_{\rm reg}}X}$ in $T^*X$ given by the conormal varieties of certain irreducible closed subvarieties $Z \subseteq X$. 

Since the characteristic cycle of a perverse sheaf is known to be effective (e.g., see \cite[Corollary 5.2.24]{Di2}), the positivity of $\chi(X,\sP)$ follows immediately from the ampleness of the cotangent bundle of $X$ together with Theorem \ref{thm_positive}. 
\end{proof}

Since compact ball quotients have ample cotangent bundles (e.g., see \cite[Construction 6.3.36]{La}), we get the following:
\bc\label{cor_ball}
If $X$ is a compact ball quotient and $\sP$ is a nonzero perverse sheaf on $X$, then $\chi(X,\sP)>0$.
\ec

The analogous result of Fulton-Lazarsfeld for nef vector bundles is proved by Demailly-Peternell-Schneider \cite{DPS}. 
\begin{thm}{\rm (\cite[Proposition 2.3]{DPS})}\label{thm_nef}
Let $X$ be a projective manifold and let $E$ be a rank $r$ nef vector bundle on $X$. For any $r$-dimensional conic subvariety $C$ of $E$, the intersection number satisfies
\[
\langle C, Z_E\rangle_E\geq 0
\]
where $Z_E$ is the zero section of $E$. 
\end{thm}

The following analog of Proposition \ref{pr1} can be proved by the same arguments. 
\begin{prop}\label{prop_nef}
Let $X$ be a projective manifold with nef cotangent bundle, and let $\sP$ be a nonzero perverse sheaf on $X$. Then $\chi(X, \sP)\geq 0$. In particular, $(-1)^{\dim X}\chi(X) \geq 0$.\end{prop}
Since the intersection complex $\mathrm{IC}_Z$ of any pure-dimensional complex algebraic variety $Z$ is a perverse sheaf and since $\mathrm{IH}^k(Z)\cong \H^{k+\dim Z}(Z, \mathrm{IC}_Z)$, we have the following. 
\begin{cor}\label{cor_subvariety}
Let $X$ be a projective manifold with nef cotangent bundle, and let $Z$ be an irreducible closed subvariety of $X$. Then the intersection cohomology Euler characteristics of $Z$, that is, 
\[
\chi_{\mathrm{IH}}(Z)\coloneqq \sum_{0\leq k\leq 2\dim Z} (-1)^k \dim \mathrm{IH}^k(Z)
\]
satisfies 
\[
(-1)^{\dim Z}\chi_{\mathrm{IH}}(Z)\geq 0. 
\]
\end{cor}

\bex\rm
The class of complex projective manifolds whose cotangent bundles are nef is closed under taking products, subvarieties and finite unramified covers, and it includes smooth subvarieties of abelian varieties. It should also be noted that if $A$ is an abelian variety of dimension $g$, and $X \subset A$ is a smooth subvariety of dimension $n$ and codimension $g-n<n$, then the cotangent bundle of $X$ is not ample (see \cite{Deb, Sch}, and also, \cite[Example 7.2.3]{La}). On the other hand, for an arbitrary smooth $m$-dimensional projective variety $M$ and each $n \leq m/2$, there exist plenty of smooth $n$-dimensional subvarieties $X \subset M$ with ample cotangent bundle, e.g., complete intersection of sections of $M$ by general hypersurfaces of sufficiently high degrees in the ambient projective space, see \cite{BD,X}. 
\eex


\section{Locally constant constructible complexes}\label{locc}
Recall that a proper map $f:X \to Y$ of smooth complex algebraic varieties is a $\bZ$-homology fiber bundle if and only if the higher derived pushforwards $R^if_*\Z_X$ are local systems on $Y$ for all $i \geq 0$. If this is the case, we say that the constructible complex $Rf_*\Z_X$ is locally constant. 
This motivates the following definition (see also \cite{Bo, GM}).

\begin{defn}\rm
Let $M$ be a complex manifold, $\dR$ a commutative ring, and let $\sF$ be a bounded $\dR$-constructible complex. We say that $\sF$ is {\it locally constant} if the cohomology sheaves $\sH^k(\sF)$ are local systems for all $k$. 
\end{defn}

Since the subcategory of local systems is a weak Serre subcategory of the abelian category of constructible sheaves, we have the following Lemma (e.g., see \cite[Section 13.17]{St}). 
\begin{lemma}\label{lemma_cone}
The mapping cone of any morphism of locally constant bounded $\dR$-constructible complexes is also locally constant. 
\end{lemma}

\begin{prop}\label{prop_constant}
Let $\sF$ be a bounded $\dR$-constructible complex on a complex manifold $M$. The following conditions are equivalent:
\begin{enumerate}
\item $\sF$ is locally constant,
\item the perverse cohomology sheaves $^p\sH^k(\sF)$ are shifts of local systems for all $k$.
\end{enumerate}
\end{prop}
\begin{proof}
Since $\sF$ is an iterated extension of $^p\sH^k(\sF)[-k]$, $k\in \Z$, the implication $(2)\Rightarrow (1)$ follows from Lemma \ref{lemma_cone}. Conversely, if $\sH^k(\sF)$ are local systems for all $k$, then $\sF$ is constructible with respect to the trivial stratification of $M$, that is, the one with a single stratum. 
The truncations $^p\tau^{\leq k}$ (resp.  $^p\tau^{\geq k}$) and $\tau^{\leq k-\dim M}$ (resp., $\tau^{\geq k-\dim M}$) are then equal, where, $({^p\tau}^\leq, {^p\tau}^\geq)$ denotes the perverse truncation. Thus, 
\[
^p\sH^k(\sF)\cong \sH^{k-\dim M}(\sF)[\dim M]
\]
is the shift of a local system. 
\end{proof}


\section{The rational and integral Bobadilla-Koll\'ar questions}\label{sec5}
In this section, we answer positively the rational Bobadilla-Koll\'ar question for compact ball quotients and the integral Bobadilla-Koll\'ar question for abelian varieties. The rational version is easier to handle thanks to the decomposition theorem. We will first simultaneously answer the rational Bobadilla-Koll\'ar question for both ball quotients and abelian varieties. For the integral version, we have to work with fields of positive characteristics, in which case neither the decomposition theorem nor the stronger generic vanishing theorem \cite[Theorem 1.3]{BSS} are available. We use characteristic cycles and the geometry of abelian varieties to prove a key non-vanishing property of the Mellin transformation, which is sufficient to answer the integral  Bobadilla-Koll\'ar question for abelian varieties. 

Let $Y$ be an aspherical compact complex manifold with fundamental group $G$ and of dimension $d$. Let $\sF$ be a $\K$-constructible complex on $Y$, where $\K$ is a field, and assume that each cohomology group $$V^j:=\H^j(\sM_!(\sF))$$ is a finite dimensional $\K$-vector space.  Notice that $V^j$ has a natural $\K[G]$-module structure, and we denote the corresponding $\K$-local system on $Y$ by $L^j$. 
\begin{lemma}\label{lemma_morphism}
Under the above notations and assumptions, let $V^r=\H^r(\sM_!(\sF))$ be the  nonzero cohomology group of the lowest degree. Then there exists a morphism $\phi: L^r[2d-r]\to \sF$ such that the induced morphism $\sM_!(\phi): \sM_!(L^r[2d-r])\to \sM_!(\sF)$ induces an isomorphism
\[
\H^r\Big( \sM_!\big(L^r[2d-r]\big)\Big)\cong \H^r\big( \sM_!(\sF)\big). 
\]
\end{lemma}
\begin{proof}
The canonical truncation on $\sM_!(\sF)$ induces a spectral sequence 
\[
E^{ij}_2=\H^i\Big(\H^j\big(\sM_!(\sF)\big)\Lotimes_{\K[G]}(V^r)^\vee\Big)\Longrightarrow \H^{i+j}\Big(\sM_!(\sF)\Lotimes_{\K[G]}(V^r)^\vee\Big).
\]
In fact, choosing a free resolution $F^\bullet$ of $(V^r)^\vee$ and a complex $G^\bullet$ representing $\sM_!(\sF)$, it is the spectral sequence of the double complex $G^\bullet\otimes_{\K[G]} F^\bullet$. 

By assumption, $\H^j(\sM_!(\sF))=0$ if $j<r$. By Example \ref{ex9} and Proposition~\ref{prop_iso1}, we have
\[
\H^i\Big(\H^j\big(\sM_!(\sF)\big)\Lotimes_{\K[G]}(V^r)^\vee\Big)=\H^i\Big(V^j\Lotimes_{\K[G]}(V^r)^\vee\Big)\cong \H^{i+2d}\Big(Y, L^j\otimes_\K (L^r)^\vee \Big),
\]
which vanishes when $i<-2d$. Therefore, 
\[
E_2^{-2d, r}\cong E_\infty^{-2d, r} \cong \H^{r-2d}\Big(\sM_!(\sF)\Lotimes_{\K[G]}(V^r)^\vee\Big),
\]
that is, 
\[
\H^{-2d}\Big(V^r\Lotimes_{\K[G]}(V^r)^\vee\Big)\cong  \H^{r-2d}\Big(\sM_!(\sF)\Lotimes_{\K[G]}(V^r)^\vee\Big).
\]
Since $\sM_!(L^r)\cong V^r[-2d]$ (cf. Example \ref{ex9}), by Proposition~\ref{prop_iso1} we have canonical isomorphisms
\[
\H^0\big(Y, L^r\otimes_\K (L^r)^\vee\big)\cong \H^0\Big(\sM_!(L^r)\Lotimes_{\K[G]} (V^r)^\vee\Big)\cong \H^{-2d}\Big(V^r\Lotimes_{\K[G]}(V^r)^\vee\Big).
\]
On the other hand, also by Proposition \ref{prop_iso1}, we have canonical isomorphisms
\[
\mathrm{Hom}(L^r[2d-r], \sF)\cong \H^{r-2d}(Y, \sF\otimes (L^r)^\vee)\cong \H^{r-2d}\Big(\sM_!(\sF)\Lotimes_{\K[G]}(V^r)^\vee\Big).
\]
Combining the above three displayed equations, we have a natural isomorphism
\[
\H^0\big(Y, L^r\otimes_\K (L^r)^\vee\big)\cong \mathrm{Hom}(L^r[2d-r], \sF).
\]
There exists an element in $\H^0(Y, L^r\otimes_\K (L^r)^\vee)$ associated to the identity map on $L^r$. We let $\phi$ be the corresponding element in $\mathrm{Hom}(L^r[2d-r], \sF)$. Then $\phi$ satisfies the desired property. 
\end{proof}

\begin{prop}\label{prop_finite1}
Let $Y$ be an aspherical projective manifold of dimension $d$, and let $\K$ be any field. 
 Let $\sP$ be a simple $\K$-perverse sheaf on $Y$. Suppose that $\sM_!(\sP)$ is nonzero and the cohomology of $\sM_!(\sP)$ in every degree is finite dimensional. Then $\sP$ is a shift of a local system. 
\end{prop}
\begin{proof}
As in Lemma \ref{lemma_morphism}, we let $\H^r(\sM_!(\sP))$ be the nonzero cohomology group of the lowest degree, with corresponding local system on $Y$ denoted by $L^r$. By  Lemma \ref{lemma_morphism} there exist a nonzero morphism $\phi: L^r[2d-r]\to \sP$. By definition, $\sM_!(\sP)=R\widetilde{q}_!(p^*\sP)$, where $p: \widetilde{Y}\to Y$ is the universal covering map and $\widetilde{q}: \widetilde{Y}\to pt$ is the projection to a point. Since $p^*\sP$ is a perverse sheaf on $\widetilde{Y}$ (cf. \cite[Proposition 5.2.13 and Corollary 5.2.15]{Di2}), the cohomology 
\[
\H^k\big(R\widetilde{q}_!(p^*\sP)\big)\cong \H_c^k\big(\widetilde{Y}, p^*\sP\big)
\]
is zero if $k\notin [-d, d]$ (e.g., see \cite[Proposition 5.2.20]{Di2}). Thus, we have $r\leq d$. On the other hand, since \[
L^r[2d-r]\in \,^p D_c^{\leq r-d}(Y, \K)\quad \textrm{and\quad} \sP\in  \,^p D_c^{\geq 0}(Y, \K),
\]
the existence of a nonzero morphism $\phi: L^r[2d-r]\to \sP$ implies that $r-d\geq 0$. Combining the two inequalities, we have that $r=d$. Thus, we have a nonzero morphism of perverse sheaves $L^r[d]\to \sP$. Since $\sP$ is simple, it must be a quotient of $L^r[d]$ in the category of perverse sheaves. Therefore, $\sP$ is the shift of a local system. 
\end{proof}
\begin{thm}\label{th23}
Let $f : X \to Y$ be a morphism of smooth projective varieties.
Let $\widetilde{Y}$ be the universal cover of $Y$, and assume that $\widetilde{X}\coloneqq X\times_{Y} \widetilde{Y}$  is homotopy equivalent to a finite CW-complex. Suppose that $Y$ is either an abelian variety or an aspherical projective manifold with an ample cotangent bundle. 
Then $f$ is a $\Q$-homology fiber bundle.
\end{thm}
\begin{proof}
By our assumptions, the cohomology groups $\H^k(\widetilde{X}, \Q)$ are finite dimensional $\Q$-vector spaces. 
The assertion in the theorem is equivalent to showing that $Rf_*\Q_X$ is locally constant.

Since $f$ is proper, by Poincar\'e duality and proper base change (e.g., see \cite[Theorem 5.1.7]{Ma}) we have 
\[
\H^{2\dim X-k}\big(\widetilde{X}, \Q\big)^\vee\cong \H_c^k\big(\widetilde{X}, \Q\big)\cong \H_c^k\big(\widetilde{Y}, R\widetilde{f}_!\Q_{\widetilde{X}}\big)\cong \H^k_c\big(\widetilde{Y}, p^*R{f}_!\Q_{X}\big) \cong \H^k_c\big(\widetilde{Y}, p^*R{f}_*\Q_{X}\big)
\]
where $\widetilde{f}: \widetilde{X}\to \widetilde{Y}$ is the lifting of $f: X\to Y$, and $p: \widetilde{Y}\to Y$ is the universal covering map. Thus, the vector spaces $\H^k_c(\widetilde{Y}, p^*(R{f}_*\Q_{X}))$ are finite dimensional for all $k$. 
Applying the decomposition theorem \cite{BBD} for the proper map $f:X \to Y$ yields is a decomposition
\[
R{f}_*\Q_{X}\cong \bigoplus_{1\leq i\leq l} \sP_i[n_i]
\]
where $n_i\in \Z$ and $\sP_i$ are nonzero simple perverse sheaves. Therefore, each $\sP_i$ has the property that $\H^k_c(\widetilde{Y}, p^*(\sP_i))$ are finite dimensional for all $k$. 

By the definition of Mellin transformation, we have
\[
\H^k_c\big(\widetilde{Y}, p^*\sP_i\big)\cong \H^k\big(\sM_!(\sP_i)\big).
\]
Thus, the cohomology groups of $\sM_!(\sP_i)$ are all finite dimensional $\Q$-vector spaces. To apply Proposition \ref{prop_finite1}, we need to prove that $\sM_!(\sP_i)$ are nonzero. When $Y$ is an aspherical projective manifold with an ample cotangent bundle (e.g., a compact ball quotient), this follows from Corollary \ref{cor_nonzero} and Proposition  \ref{pr1}. When $Y$ is an abelian variety, this follows from the Riemann-Hilbert correspondence and the Fourier-Mukai transformation being an equivalence of categories (\cite{Lau} and \cite{Ro}). See \cite[Theorem 2]{We} and \cite[Theorem 7.6]{Sc15} for stronger results. 

Now, by Proposition \ref{prop_finite1}, each $\sP_i$ is a shifted local system, and hence $R{f}_*\Q_{X}\cong \bigoplus_{1\leq i\leq l} \sP_i[n_i]$ is locally constant. 
\end{proof}

In the remainder of this section, we deal with the integral Bobadilla-Koll\'ar question for abelian varieties. 

\begin{lemma}\label{lemnew}
Let $Y$ be an aspherical complex manifold of dimension $d$, and fix a field $\K$. Let $\sF$ be a locally constant $\K$-constructible complex with $\sM_!(\sF)=0$. Then $\sF=0$.
\end{lemma}
\begin{proof}
The canonical filtration on $\sF$ induces a spectral sequence 
\[
E_2^{p, q}=H^p\big(\sM_!\big(\sH^q(\sF)\big)\big)\Rightarrow H^{p+q}\big(\sM_!(\sF)\big).
\]
By Example \ref{ex9}, the assumption that $\sH^p(\sF)$ are local systems for all $p$ implies that the spectral sequence degenerates at the second page, and 
\[
H^{k}\big(\sM_!(\sF)\big)\cong H^{2d}\big(\sM_!\big(\sH^{k-2d}(\sF)\big)\big)\cong V^{k-2d},
\]
where $V^{k-2d}$ is the stalk of $\sH^{k-2d}(\sF)$ at the base point. Therefore, $\sH^{k-2d}(\sF)=0$ for all $k$, that is, $\sF=0$. 
\end{proof}

\begin{prop}\label{prop_finite2}
Let $Y$ be an aspherical projective manifold of complex dimension $d$, and fix a field $\K$. The following statements are equivalent:
\begin{enumerate}
\item The Mellin transformation of any nonzero $\K$-constructible complex  is nonzero. 
\item For any constructible complex $\sF$ on $Y$, if $H^k(\sM_!(\sF))$ is finite dimensional for all $k$, then $\sF$ is locally constant. 
\end{enumerate}
\end{prop}

\begin{proof}
The implication $(2)\Rightarrow (1)$ follows directly from Lemma \ref{lemnew}. 

To prove the implication $(1)\Rightarrow (2)$, we use induction on the total dimension 
\[
\sigma(\sF)\coloneqq \sum_{k\in \Z}\dim_\K \H^k\big(\sM_!(\sF)\big).
\]
If $\sigma(\sF)=0$, that is, $\sM_!(\sF)=0$, then statement (1) implies that $\sF=0$. Assume that $\sigma(\sF)>0$. Let $\H^r(\sM_!(\sF))$ be the nonzero cohomology group of the lowest degree. By Lemma \ref{lemma_morphism}, there exists a morphism $\phi: L^r[2d-r]\to \sF$ such that the mapping cone $C(\phi)$ satisfies
\[
\H^k\Big(\sM_!\big(C(\phi)\big)\Big)=\begin{cases}
0 &\mbox{if }k\leq r\\
\H^k(\sM(\sF)) &\mbox{if } k>r. 
\end{cases}
\]
Thus, $\sigma(C(\phi))<\sigma(\sF)$. By the inductive hypothesis, $C(\phi)$ is locally constant. By Lemma~\ref{lemma_cone} and the distinguished triangle
\[
L^r[2d-r]\to \sF\to C(\phi)\to L^r[2d-r+1],
\]
we know that $\sF$ is also locally constant. 
\end{proof}

We also have the following result, whose proof will be given later on.
\begin{prop}\label{prop_nonzero}
Let $A$ be an abelian variety, and let $\sF$ be a nonzero $\K$-constructible complex. Then the Mellin transformation $\sM_!(\sF)$ is nonzero. 
\end{prop}

Combining the above two propositions, we have the following corollary. 
\begin{cor}\label{c26}
Let $A$ be an abelian variety, and let $\sF$ be a $\K$-constructible complex. If $\sM_!(\sF)$ has finite dimensional cohomology in every degree, then $\sF$ is locally constant. 
\end{cor}

\begin{lemma}\label{lemma_algebra}
Let $f: M^\bullet\to N^\bullet$ be a homomorphism of bounded complexes of free $\Z$-modules, with finitely generated cohomology. If $f\otimes_\Z \K: M^\bullet\otimes_\Z \K\to N^\bullet\otimes_\Z \K$ is a quasi-isomorphism for any field $\K$, then $f$ is a quasi-isomorphism. 
\end{lemma}
\begin{proof}
Consider the mapping cone $C(f)$ of $f$. By assumption, $C(f)\otimes_\Z \K$ is acyclic for any field $\K$. Suppose that $C(f)$ is not acyclic. Let $\H^k(C(f))$ be a nonzero cohomology group. Since $\H^k(C(f))$ is a finitely generated abelian group, we can choose $\K$ to be either $\Q$ or a finite field such that $\H^k(C(f))\otimes_\Z \K\neq 0$. By the universal coefficient theorem, we have a short exact sequence
\[
0\to \H^k(C(f))\otimes_\Z \K\to \H^k(C(f)\otimes_\Z \K)\to \mathrm{Tor}\big(\H^{k+1}(C(f)), \K\big)\to 0.
\]
This contradicts the fact that  $C(f)\otimes_\Z \K$ is acyclic. Hence $C(f)$ must be acyclic, that is, $f$ is a quasi-isomorphism. 
\end{proof}

\begin{lemma}\label{lemma_KZ}
Let $\sF$ be a $\Z$-constructible complex on a complex manifold $M$. If $\sF \Lotimes_\Z \K$ is  locally constant for any field $\K$, then $\sF$ is locally constant. 
\end{lemma}
\begin{proof}
First, we can take an open cover $\{U_\lambda\}_{\lambda\in I}$ such that each $U_\lambda$ is contractible and the cohomology groups of $R\Gamma(U_\lambda, \sF)$ are finitely generated. The second condition can be achieved by choosing $U_\lambda$ as sufficiently small balls. 

We claim that, for any $U_\lambda$, there exists a canonical isomorphism
\[
R\Gamma(U_\lambda, \sF)\Lotimes_\Z \K\cong R\Gamma(U_\lambda, \sF\Lotimes_\Z \K)
\]
in the derived category $D(\K)$ of complexes of $\K$-vector spaces. In fact, taking an injective resolution $\sI^\bullet$ of $\sF$ and a free resolution $P^\bullet$ of $\K$, the total complex of $\sI^\bullet\otimes_\Z P^\bullet$ is a complex of injective sheaves representing $\sF\Lotimes_\Z \K$, and hence both sides are isomorphic to the total complex in $D(\K)$. Since taking direct limit commutes with taking tensor product, using the same resolutions, we also have a canonical isomorphism
\[
  i_x^*\sF\Lotimes_\Z \K\cong  i_x^*(\sF\Lotimes_\Z \K)
\]
in $D(\K)$ for any point $x\in M$. 

Since $\sF\Lotimes_\Z \K$ is locally constant, the restriction map
\[
R\Gamma(U_\lambda, \sF\Lotimes_\Z \K)\to i_x^*(\sF\Lotimes_\Z \K)
\]
is an isomorphism. Combining the above three displayed equations, it follows that for any field $\K$, the restriction map $R\Gamma(U_\lambda, \sF)\to i_x^*\sF$ induces isomorphisms
\[
R\Gamma(U_\lambda, \sF)\Lotimes_\Z \K \to i_x^*\sF\Lotimes_\Z \K. 
\]
Since both $R\Gamma(U_\lambda, \sF)$ and $i_x^*\sF$ are bounded complexes with finitely generated cohomology groups, it follows from Lemma \ref{lemma_algebra} that $R\Gamma(U_\lambda, \sF)\to i_x^*\sF$  is an isomorphism. 
\end{proof}

\begin{cor}\label{cor_integral}
Let $A$ be an abelian variety, and let $\sF$ be a $\Z$-constructible complex. If the cohomology group of $\sM_!(\sF)$ in every degree is a finitely generated abelian group, then $\sF$ is locally constant.
\end{cor}
\begin{proof}
Since $\sM_!(\sF\Lotimes_\Z \K) \cong \sM_!(\sF) \Lotimes_\Z \K$, our assumption implies that $\sM_!(\sF\Lotimes_\Z \K)$ has finite dimensional cohomology groups.  So $\sF\Lotimes_\Z \K$ are locally constant by Corollary \ref{c26}. Now the assertion follows from Lemma \ref{lemma_KZ}. 
\end{proof}

The rest of this section is devoted to proving Proposition \ref{prop_nonzero} and to answer the integral homology version of the Bobadilla-Koll\'ar question for abelian varieties. To this end, we will use induction to reduce to the case when $A$ is a simple abelian variety. We first need the following result about simple abelian varieties.

\begin{prop}\label{cor_simple}
Let $A$ be a simple abelian variety. Let $\sP$ be a $\K$-perverse sheaf that is not locally constant. Then $\chi(A,\sP)>0$. 
\end{prop}
\begin{proof}
This is a direct consequence of \cite[Proposition 10.1]{KW}. Note that even though the result in loc. cit. is stated with complex coefficients, it works over any field, since the proof only uses characteristic cycles. 
\end{proof}

We also recall a generic  vanishing theorem of Bhatt-Schnell-Scholze. (See also \cite{Kra, LMWb, LMWd} for a generalization to semi-abelian varieties.) 
\begin{thm}[\cite{BSS}]\label{thm_BSS}
Let $\sP$ be a $\K$-perverse sheaf on an abelian variety $A$. Let $\overline{\K}$ be the algebraic closure of $\K$. For a general rank one $\overline{\K}$-local system $L$ on $A$, one has
\[
\H^i(A, \sP\otimes_\K L)=0\quad \textrm{for all $i\neq 0$. }
\]
\end{thm}

We now have all the ingredients to complete the proof of Proposition \ref{prop_nonzero}.
\begin{proof}[Proof of Proposition \ref{prop_nonzero}] 
Let $\sF$ be a $\K$-constructible complex on an abelian variety $A$ such that $\sM_!(A,\sF)=0$. We will show that $\sF=0$. 

If $\sF$ is locally constant, then Lemma \ref{lemnew} implies that $\sF=0$.

Let us next assume that $A$ is a simple abelian variety and $\sF$ is not locally constant. Then by Proposition \ref{prop_constant}, there exists some perverse cohomology $^p\sH^l(\sF)$ that is not locally constant. By Corolllary \ref{cor_simple}, we have $\chi(A,{^p\sH}^l(\sF))>0$. Let $L$ be a general rank one $\overline{\K}$-local system on $A$. Then by Theorem \ref{thm_BSS}, we have
\[
\H^j\big(A, \,^p\sH^i(\sF)\otimes_\K L\big)=0\quad \textrm{for all $i\in \Z$ and $j\neq 0$. }
\]
Since $\chi(A,{^p\sH}^l(\sF)\otimes_\K L)=\chi(A,{^p\sH}^l(\sF))>0$, we have that 
$\H^0(A, \,^p\sH^l(\sF)\otimes_\K L)\neq 0$. Since
\[
^p\sH^i(\sF)\otimes_\K L\cong\, ^p\sH^i(\sF \otimes_\K L),
\]
the perverse cohomology spectral sequence
\[
E^{ij}_2=\H^i\big(A, \,^p\sH^j(\sF\otimes_\K L)\big)\Longrightarrow \H^{i+j}(A, \sF\otimes_\K L)
\]
degenerates at the $E_2$-page. Therefore, $\H^l(A, \sF\otimes_\K L)\neq 0$. By Corollary \ref{cor_allzero}, we have that $\sM_!(A,\sF)\neq 0$, a contradiction to the assumption. So this case cannot occur. 

So far, we have proved the proposition when $A$ is a simple abelian variety. 
Finally, we prove the general case using induction on the dimension of $A$. Since we are done with the case when $A$ is simple, from now on we can assume that $A$ is not simple. In this case, there exists a short exact sequence of positive dimensional abelian varieties
\[
0\to A_1\to A\xrightarrow{p_2} A_2\to 0.
\]
In general, the short exact sequence does not split in the category of abelian varieties, but it does split in the category of real Lie groups. Fix such a splitting, and denote the induced projection $A\to A_1$ by $p_1$. 

Since $\sM_!(A,\sF)=0$, we get by Corollary \ref{cor_allzero} that 
\[
H^i\big(A, \sF\otimes_\K p_1^*L_1\otimes_{\overline{\K}} p_2^*L_2\big)=0
\]
for and $i\in \Z$ and any rank one $\overline{\K}$-local systems $L_1$ and $L_2$ on $A_1$ and $A_2$, respectively. Moreover, by the projection formula, we have
\[
0=H^i\big(A, \sF\otimes_\K p_1^*L_1\otimes_{\overline{\K}}  p_2^*L_2\big)\cong H^i\big(A_{2}, Rp_{2*}(\sF\otimes_\K p_1^*L_1)\otimes_{\overline{\K}}  L_2\big).
\]
By Corollary \ref{cor_allzero}, we have
\[
\sM_!(A_2, Rp_{2*}\big(\sF\otimes_\K  p_1^*L_1)\big)=0.
\]
By the induction hypothesis, the proposition holds for $A_2$. So we have that 
\[
Rp_{2*}(\sF\otimes_\K p_1^*L_1)=0\quad \textrm{for any rank one $\overline{\K}$-local system $L_1$ on $A_1$.}
\]
Choose any point $x\in A_2$. By the base change formula, we have
\[
0 =i_x^*Rp_{2*}(\sF\otimes_\K p_1^*L_1)\cong Rp_{2*}\Big((\sF\otimes_\K p_1^*L_1)|_{p_2^{-1}(x)}\Big) \cong {R}p_{2*}\Big(\sF|_{p_2^{-1}(x)}\otimes_\K p_1^*L_1|_{p_2^{-1}(x)}\Big), 
\]
or, equivalently,
\[
H^i\Big(p_2^{-1}(x), \sF|_{p_2^{-1}(x)}\otimes_\K p_1^*L_1|_{p_2^{-1}(x)}\Big)=0 
\]
for any $i\in \Z$ and any rank one $\overline{\K}$-local system $L_1$ on $A_1$.

Notice that $p_2^{-1}(x)$ is isomorphic to $A_1$ and as $L_1$ varies through all rank one $\overline{\K}$-local systems on $A_1$, $p_1^*L_1|_{p_2^{-1}(x)}$ varies through all rank one $\overline{\K}$-local systems on $p_2^{-1}(x)$. Thus, by Corollary \ref{cor_allzero}, we have
\[
\sM_!\Big(p_2^{-1}(x), \sF|_{p_2^{-1}(x)}\Big)=0.
\]
Again, by the induction hypothesis, this implies that $\sF|_{p_2^{-1}(x)}=0$. Since $x$ is an arbitrary point on $A_2$, we  conclude that $\sF=0$. 
\end{proof}

We can now prove the integral Bobadilla-Koll\'ar question for abelian varieties.
\begin{thm}\label{th38}
Let $f : X \to Y$ be a morphism from a smooth projective variety to an abelian variety.
Let $\widetilde{Y}$ be the universal cover of $Y$, and assume that $\widetilde{X}\coloneqq X\times_{Y} \widetilde{Y}$  is homotopy equivalent to a finite CW-complex. Then $f$ is a $\Z$-homology fiber bundle.
\end{thm}

\begin{proof} 
It follows by our assumptions, as in the proof of Theorem \ref{th23}, that  the Mellin transformation $\sM_!(Rf_*\Z_X)$  has finitely generated cohomology groups. By Corollary \ref{cor_integral}, we get that 
 $Rf_*\Z_X$ is locally constant. By definition, this is equivalent to the fact that the map $X\to Y$ is a $\Z$-homology fiber bundle. 
\end{proof}
\begin{remark}\rm
The same proof also works for a compact complex torus. Notice that for a non-simple abelian variety, we never used the fact that, up to an isogeny, it is the product of two smaller abelian varieties. We only used the fact that it is the extension of two smaller abelian varieties, which also holds in the category of compact complex tori. 
\end{remark}

\begin{remark}\rm
In \cite{GL}, Gabber and Loeser proved that on a complex affine torus the Mellin transformation of any nontrivial constructible complex is nonzero. In fact, using inductive arguments as in \cite[Theorem 4.3]{LMWd}, we can extend Proposition \ref{prop_nonzero} to semiabelian varieties. This result can be used to give a positive answer to the integral Bobadilla-Koll\'ar question when $Y$ is a semi-abelian variety. 
\end{remark}

Applying Theorem \ref{th38} to the Albanese map of a complex projective manifold, we have the following. 
\begin{cor}
Let $X$ be a projective manifold. Let $X^{ab}$ be the universal free abelian cover of $X$, that is, the covering space of $X$ associated to the group homomorphism $\pi_1(X)\to H_1(X, \Z)/\mathrm{torsion}$. If $X^{ab}$ is homotopy equivalent to a finite CW-complex, then the Albanese map of $X$ is a $\Z$-homology fiber bundle. 
\end{cor}


\section{Aspherical projective manifolds and the Singer-Hopf conjecture}\label{conj}
This paper is motivated in part by the following long-standing conjecture (e.g., see  \cite[Conjecture 25.1]{G}):

\begin{conj}{\rm (Singer-Hopf)}\label{SH}
Suppose $X^{2n}$ is a closed, aspherical manifold of real dimension $2n$. Then $$(-1)^n  \chi(X^{2n}) \geq 0.$$
\end{conj}

The conjecture is true for $n=1$ (i.e., real dimension $2$) since the only closed surfaces with positive Euler characteristic are $S^2$ and $\R P^2$, and they are the only non-aspherical ones. In the special case when $X^{2n}$ is a Riemannian manifold with non-positive sectional curvature, this conjecture is attributed to Hopf and Chern. (The fact that a Riemannian manifold with non-positive sectional curvature is aspherical is a consequence of Hadamard's Theorem.) It was strengthened to the aspherical case by Singer, and it also asserts the vanishing of all $L^2$-Betti numbers of the universal cover, except possibly the middle one. 

Jost and Zuo \cite{JZ} proved Conjecture \ref{SH} for $X$ a compact K\"ahler manifold with non-positive sectional curvature. Their techniques rely on analytic arguments introduced by Gromov \cite{Gro}, who confirmed the Singer-Hopf Conjecture for K\"ahler hyperbolic manifolds (these include K\"ahler manifolds with negative and pinched sectional curvature).

We propose the following natural generalization of Conjecture \ref{SH} in the projective context:
\begin{conj} \label{co2} If $X$ is an aspherical  projective manifold and $\mathcal{P}$ is a perverse sheaf on $X$, then the Euler characteristic of $\mathcal{P}$  is non-negative, that is, $\chi(X,\mathcal{P}) \geq 0$.
\end{conj}

By Proposition \ref{prop_nef}, Conjecture \ref{co2} is a consequence of the following two conjectures.

\begin{conj}\label{conj1}
Let $Y$ be a projective manifold. If the universal cover of $Y$ is a Stein manifold, then the cotangent bundle of $Y$ is nef. 
\end{conj}
A weaker version of this conjecture is proved in \cite{Kr}, namely, if the universal cover of $Y$ is a bounded domain in a Stein manifold, then the cotangent bundle of $Y$ is nef. 

\begin{conj}\label{conj2}
If $Y$ is an aspherical projective manifold, then the universal cover is Stein. 
\end{conj}

\begin{remark}\rm
Corollary \ref{cor_subvariety} suggests that one may formulate a generalization of the Singer-Hopf conjecture to singular varieties. For example, one may conjecture that if the universal cover of a (possibly singular) complex projective variety $X$ is a Stein space, then $(-1)^{\dim X}\chi_{\mathrm{IH}}(X)\geq 0$. 
\end{remark}

We show below that Conjecture \ref{conj2} follows from the following Shafarevich conjecture (see \cite{E} for an introduction) on the universal cover of projective manifolds.
\begin{conj}{\rm (Shafarevich conjecture)}
The universal cover of any projective manifold is holomorphically convex. 
\end{conj}

\begin{prop}
Suppose $Y$ is an aspherical compact projective manifold. Then its universal cover $\widetilde{Y}$ does not contain any positive dimensional compact analytic subvariety. 
\end{prop}
\begin{proof}
Suppose that $\widetilde{Y}$ does contain a positive dimensional compact analytic subvariety $Z$. By taking intersections with the preimage of general hyperplane sections and taking irreducible components, we can assume that $Z$ is one-dimensional and irreducible. Let $Z'$ be the normalization of $Z$. Then $Z'$ is a compact Riemann surface. Consider the composition
\[
Z'\to Z\hookrightarrow \widetilde{Y}\to Y,
\]
where the first map is the normalization map, the second is the inclusion map and the third is the universal covering map. The composition is a nonconstant holomorphic map, whose image is an irreducible 1-cycle in $Y$. In a projective manifold, any positive cycle corresponds to a nonzero element in the homology groups. So the above composition is not null-homotopic. However, since it factors through a contractible space $\widetilde{Y}$, it must be null-homotopic, a contradiction. 
\end{proof}
\begin{cor}\label{cor_twoconj}
The Shafarevich conjecture implies Conjecture \ref{conj2}.
\end{cor}
\begin{proof}
Let $Y$ be an aspherical projective manifold with universal cover $\widetilde{Y}$. The Shafarevich conjecture implies that $\widetilde{Y}$ is holomorphically convex. By the Cartan-Remmert reduction, there exists a proper surjective holomorphic map $f: \widetilde{Y}\to Z$ to a Stein space with connected fibers such that $f_*(\sO_Y)\cong \sO_Z$. Since $\widetilde{Y}$ does not contain any  positive dimensional compact analytic subvariety, the map $f$ must be a bijection, and hence a biholomorphic map. 
\end{proof}

If the fundamental group of $Y$ admits a faithful finite-dimensional linear representation, 
the Shafarevich conjecture is proved in the recent breakthrough \cite{EKPR} by Eyssidieux-Katzarkov-Pantev-Ramachandran. We are not aware of any example of an aspherical projective manifold whose fundamental group does not admit a finite-dimensional faithful representation.  This leads us to the following question.
\begin{question}
Does the fundamental group of an aspherical projective manifold always admit a finite-dimensional faithful representation?
\end{question}



\begin{thebibliography}{ABCD}

\bibitem[BBD]{BBD}  Beilinson, A. A.,  Bernstein, J.,  Deligne, P.,  {\it Faisceaux pervers}, Analysis and topology on singular spaces, I (Luminy, 1981), 5--171, Ast\'erisque {100}, Soc. Math. France, Paris, 1982.

\bibitem[BK]{BK}  Bobadilla, J. F., Koll\'{a}r, J.,  {\it Homotopically trivial deformations,} 
J. Singul. 5 (2012), 85--93.

\bibitem[Bo]{Bo} A. Borel; et al.: {\it Intersection cohomology}. Notes on the seminar held at the University of Bern, Bern, 1983. Reprint of the 1984 edition. Modern Birkh\"auser Classics. Birkh\"auser Boston, Inc., Boston, MA, 2008.

\bibitem[BSS]{BSS} Bhatt, B., Schnell, C., Scholze, P., {\it Vanishing theorems for perverse sheaves on abelian varieties, revisited}, Selecta Math. (N.S.) 24 (2018), no. 1, 63--84.

\bibitem[BD]{BD} Brotbek, D., Darondeau, L., 
{\it Complete intersection varieties with ample cotangent bundles}, 
Invent. Math. 212 (2018), no. 3, 913--940.

\bibitem[De]{Deb} Debarre, O., {\it Varieties with ample cotangent bundle}, Compos. Math. 141 (2005), 1445--1459; {\it Corrigendum}, Compos. Math. 149 (2013), 505--506.

\bibitem[DPS]{DPS} Demailly, J.-P., Peternell, T., Schneider, M., {\it Compact complex manifolds with numerically effective tangent bundles}, J. Algebraic Geom. 3 (1994), no. 2, 295--345. 

\bibitem[Di]{Di2} Dimca, A., \emph{Sheaves in topology}, Universitext. Springer-Verlag, Berlin, 2004.

\bibitem[Ey]{E} Eyssidieux, P., {\it Lectures on the Shafarevich conjecture on uniformization. }  Complex manifolds, foliations and uniformization, 101--148,
Panor. Synth\`eses, 34/35, Soc. Math. France, Paris, 2011.

\bibitem[EKPR]{EKPR} Eyssidieux, P., Katzarkov, L., Pantev, T., Ramachandran, M., {\it Linear Shafarevich conjecture},  Ann. of Math. (2) 176 (2012), no. 3, 1545--1581.

\bibitem[FJ]{FJ} Farrell, F. T., Jones, L. E., 
{\it Topological rigidity for compact non-positively curved manifolds}, Differential geometry: Riemannian geometry (Los Angeles, CA, 1990), 229--274,
Proc. Sympos. Pure Math., 54, Part 3, Amer. Math. Soc., Providence, RI, 1993.

\bibitem[FL]{FuLa} Fulton, W., Lazarsfeld, R., {\it Positive polynomials for ample vector bundles}, Ann. of Math. (2) 118 (1983), no. 1, 35--60. 

\bibitem[GL]{GL} Gabber, O., Loeser, F., {\it Faisceaux pervers $l$-adiques sur un tore}, Duke Math. J. 83 (1996), no. 3, 501--606. 

\bibitem[GM]{GM} M. Goresky, R. MacPherson: {\it Intersection homology. II.} Invent. Math. 72 (1983), no. 1, 77--129.

\bibitem[Gr]{Gro} Gromov, M., {\it K\"ahler hyperbolicity and $L_2$-Hodge theory},  J. Differential Geom. 33 (1991), no. 1, 263--292.

\bibitem[Gu]{G} {\it Guido's book of conjectures}, Enseign. Math. (2) 54 (2008), no. 1-2, 3--189.

\bibitem[JZ]{JZ} Jost, J., Zuo, K., \emph{Vanishing theorems for $L^2$-cohomology on infinite coverings of compact K\"ahler manifolds and applications in algebraic geometry}, Comm. Anal. Geom. {8} (2000), no. 1, 1--30.

\bibitem[KP]{KP}  Koll\'{a}r,  J.,  Pardon,  J., {\it Algebraic varieties with semialgebraic universal cover,} J. Topol. 5 (2012), 199--212.

\bibitem[Ka]{Kas} Kashiwara, M., 
{\it Index theorem for constructible sheaves},  
Differential systems and singularities (Luminy, 1983),  
Ast\'erisque No. {130} (1985), 193--209.

\bibitem[Kr]{Kra} Kr\"{a}mer, T.,  {\it Perverse sheaves on semiabelian varieties}, Rend. Semin. Mat. Univ. Padova {132} (2014), 83--102.
 
\bibitem[KW]{KW} Kr\"{a}mer, T.,  Weissauer, R., {\it Vanishing theorems for constructible sheaves on abelian varieties}, J. Algebraic Geom. {24} (2015), no. 3, 531--568.

\bibitem[Kr]{Kr} Kratz, H., {\it Compact complex manifolds with numerically effective cotangent bundles}, Doc. Math. 2 (1997), 183--193.

\bibitem[Lau]{Lau}   Laumon, G., {\it Transformation de Fourier g\'en\'eralis\'ee}, arXiv:alg-geom/9603004.

\bibitem[Laz]{La} Lazarsfeld, R., \emph{Positivity in algebraic geometry. II. Positivity for vector bundles, and multiplier ideals.}  Ergebnisse der Mathematik und ihrer Grenzgebiete {49}. Springer-Verlag, Berlin, 2004.

\bibitem[LMW1]{LMWb} Liu, Y., Maxim, L., Wang, B., {\it  Generic vanishing for semi-abelian varieties and integral Alexander modules}, 
Math. Z. 293 (2019), no. 1-2, 629--645.

\bibitem[LMW2]{LMWc} Liu, Y., Maxim, L., Wang, B., {\it  Mellin transformation, propagation, and abelian duality spaces}, Adv. Math {335} (2018), 231--260.


\bibitem[LMW3]{LMWd} Liu, Y., Maxim, L., Wang, B., {\it  Perverse sheaves on semi-abelian varieties},  Selecta Math. (N.S.) 27 (2021), no. 2, Paper No. 30, 40 pp. 

\bibitem[LMW4]{LMWe} Liu, Y., Maxim, L., Wang, B., {\it  Non-abelian Mellin transformations and applications},  arXiv:2107.05608.

\bibitem[Lu]{Lu10} L\"uck, W., {\it Survey on aspherical manifolds}, European Congress of Mathematics, 53--82, Eur. Math. Soc., Z\"urich, 2010.

\bibitem[Ma]{Ma} Maxim, L., {\it Intersection Homology \& Perverse Sheaves, with Applications to Singularities}, 
Graduate Texts in Mathematics, Vol. 281, Springer, 2019.

\bibitem[Ro]{Ro} Rothstein, M., {\it Sheaves with connection on abelian varieties}, Duke Math. J. 84 (1996), no. 3, 565--598.


\bibitem[Sc]{Sch} Schneider, M., {Symmetric differential forms as embedding obstructions and vanishing theorems}, J. Algebr. Geom. 1 (1992), 175--181.


\bibitem[Sch]{Sc15} Schnell, C., {\it Holonomic D-modules on abelian varieties}, Publ. Math. Inst. Hautes \'Etudes Sci. {121} (2015), 1--55. 

\bibitem[St]{St}  \href{https://stacks.math.columbia.edu/tag/06UP}{Stacks project}, https://stacks.math.columbia.edu/tag/06UP.


\bibitem[We]{We} Weissauer, R., {\it Degenerate Perverse Sheaves on Abelian Varieties}, arXiv:1204.2247.


\bibitem[Xi]{X} Xie, S.-Y., 
{\it On the ampleness of the cotangent bundles of complete intersections}, 
Invent. Math. 212 (2018), no. 3, 941--996.

\end{thebibliography}
\end{document}